\documentclass[12pt,twoside,reqno]{amsart}
\usepackage{amsfonts}
\usepackage{amssymb}
\usepackage[notcite,notref]{
}

\makeatletter
\@namedef{subjclassname@2010}{%
  \textup{2010} Mathematics Subject Classification}
\makeatother

\numberwithin{equation}{section}

\newtheorem{teo}{Theorem}[section]

\newtheorem{lem}[teo]{Lemma}

\theoremstyle{definition}

\newtheorem{rem}[teo]{Remark}

\numberwithin{equation}{section}

\def\a{\alpha}

\def\l{\lambda }
\def\g{\gamma }
\def\o{\omega}
\def\R{\mathbb{R}}
\def\N{\mathbb{N}}

\def\d{\delta}
\def\e{\varepsilon}
\def\t{\theta}
\def\f{\varphi}
\def\s{\sigma}
\def\loc{\operatorname{loc}}
\def\mes{\operatorname{mes}}
\def\dist{\operatorname{dist}}
\def\capa{\operatorname{cap}}

\begin{document}


\title[Limiting relations]{On limiting relations for capacities}

\author[V.I. Kolyada]{V.I. Kolyada}
\address{Department of Mathematics\\
Karlstad University\\
Universitetsgatan 1 \\
651 88 Karlstad\\
SWEDEN} \email{viktor.kolyada@kau.se}


\begin{abstract}

The paper is devoted to the study of limiting behaviour of Besov
capacities $\capa (E;B_{p,q}^\a)\,\,\, (0<\a<1)$ of sets in $\R^n$ as $\a\to 1$ or
$\a\to 0.$ Namely, let $E\subset \R^n$ and
$$J_{p,q}(\a,
E)=[\a(1-\a)q]^{p/q}\capa(E;B_{p,q}^\a).$$ It is proved that if $1\le
p<n,\,\,1\le q<\infty,$ and the set $E$ is open, then $J_{p,q}(\a,
E)$ tends to the Sobolev capacity $\capa(E;W_p^1)$ as $\a\to 1$.
This statement fails to hold for compact sets. Further, it is proved
that if the set $E$ is compact and $1\le p,q<\infty$, then
$J_{p,q}(\a, E)$ tends to $2n^p|E|$ as $\a\to 0$ ($|E|$ is the  measure of $E$). For open sets it
is not true.

\end{abstract}


\maketitle

\date{}

\maketitle

\section{Introduction}

 The Sobolev space
$W_p^1(\R^n)$ ($1\le p<\infty$) is defined as the class  of all
functions $f\in L^p({\mathbb R}^n)$ for which all first-order weak
derivatives $\partial f/\partial x_k=D_kf$ $(k=1,...,n)$ exist and
belong to $L^p({\mathbb R}^n)$.

The classical embedding theorem with limiting exponent states that
if $1\le p< n,$ then for any $f\in W_p^1(\R^n)$
\begin{equation}\label{sobolev}
||f||_{p^*}\le c \sum_{k=1}^n||D_k f||_p, \quad\mbox{where}\quad p^*=\frac{np}{n-p}.
\end{equation}
This theorem was proved by Sobolev in 1938 for $1<p<n$ and by
Gagliardo and Nirenberg in 1958 for $p=1$ (see \cite[Chapter
5]{St}).

Embeddings with limiting exponent are also true for some spaces
 defined in terms of moduli of continuity.

Let $f\in L^p({\mathbb R}^n)$ ($1\le p<\infty$) and
$k\in\{1,...,n\}.$ The partial modulus of continuity of $f$ in $L^p$
with respect to $x_k$ is defined by
$$
\o_k(f;\delta)_p =\sup_{0\le h\le \d}\left(\int_{\mathbb R^n}
|f(x+he_k)-f(x)|^p~dx\right)^{1/p}
$$
($e_k$ is the $k$th unit coordinate vector).

Let $0<\a<1,$  $1\le p<\infty$, $1\le q\le\infty,$ and $k\in\{1,...,n\}.$
 The Nikol'ski\u\i-Besov space
$B^\a_{p,q;k}({\mathbb R}^n)$ consists of all functions $f\in
L^p({\mathbb R}^n)$ such that
$$
\|f\|_{b^\a_{p,q;k}}\equiv
\Big(\int_0^{\infty}\big(t^{-\a}\o_k(f;t)_p\big)^q\,
\frac{dt}{t}\Big)^{1/q}<\infty
$$
if $q<\infty,$ and
$$
\|f\|_{b^\a_{p,\infty;k}}\equiv \sup_{t>0}t^{-\a}\o_k(f;t)_p<\infty
$$
 if $q=\infty.$ Further, set
 $$
 B_{p,q}^\a(\R^n)= \bigcap_{k=1}^n B^\a_{p,q;k}({\mathbb R}^n) \quad  \mbox{and}\quad ||f||_{b^\a_{p,q}}=\sum_{k=1}^n \|f\|_{b^\a_{p,q;k}}.
 $$
We write also $B_{p,p}^\a(\R^n)= B_{p}^\a(\R^n).$

Observe that in these definitions and notations we follow
Nikol'ski\u\i's book \cite{Nik}; they can be immediately extended to anisotropic Nikol'ski\u\i-Besov spaces.

The spaces $B^\a_{p}({\mathbb R}^n)$ are often considered as
  Sobolev spaces of fractional smoothness.
The embedding theorem with limiting exponent for these spaces
asserts that if $0<\a<1$ and $1\le p<n/\a,$ then
\begin{equation}\label{besov10}
B_p^\a(\R^n)\subset L^{p_\a}(\R^n), \quad\mbox{where}\quad p_\a=\frac{np}{n-\a p}.
\end{equation}
This theorem was proved in the late sixties independently by several
authors (for the references, see \cite[\S \, 18]{BIN}, \cite[Section
10]{K1998}).

In 2002 Bourgain, Brezis and Mironescu \cite{BBM2} discovered that
embedding $W_p^1\subset L^{p^*}$ can be obtained as the limit of embedding
(\ref{besov10}) as $\a\to 1.$ First, they proved in \cite{BBM1} that
for any $f\in W_p^1(\R^n)\,\, (1\le p<\infty)$
\begin{equation}\label{limiting1}
\lim_{\a\to 1-}(1-\a)^{1/p}||f||_{b_p^\a}\asymp||\nabla f||_p
\end{equation}
(see also \cite{Br1}, \cite[Section 14.3]{Leo1},
\cite[Section 10.2]{Maz2}). The main result in \cite{BBM2} is the
following:  if $1/2\le \a<1$ and $1\le p<n/\a,$ then for any $f\in
B_p^\a({\mathbb R}^n)$,
\begin{equation}\label{bourg}
\|f\|^p_{L^{p_\a}}\le c_n\frac{1-\a}{(n-\a p)^{p-1}}
\|f\|_{b^\a_p}^p\quad\Big(p_\a=\frac{np}{n-\a p}\Big),
\end{equation}
where a constant $c_n$ depends only on $n$. In view of
(\ref{limiting1}), inequality
(\ref{sobolev}) is a limiting case of (\ref{bourg}) as $\a\to 1-$. The proof of
(\ref{bourg}) in \cite{BBM2} was quite complicated. Afterwards,
Maz'ya and Shaposhnikova \cite{MS} gave a simpler proof of
(\ref{bourg}). Moreover, they studied the limiting behaviour of the
$B_p^\a-$norm and the sharp asymptotics of the embedding constant in
(\ref{besov10})  as $\a\to 0$. More precisely, they proved that
\begin{equation}\label{mazya}
\|f\|^p_{L^{p_\a}}\le c_{p,n}\frac{\a(1-\a)}{(n-\a p)^{p-1}}
\|f\|_{b^\a_p}^p\quad\Big(1\le p< \frac{n}{\a}, \,\,\,p_\a=\frac{np}{n-\a p}\Big).
\end{equation}
Also, it was  shown  in \cite{MS} that if $f\in B_p^{\a_0}({\mathbb
R}^n)$ for some $\a_0\in (0,1),$ then
\begin{equation}\label{mazya1}
\lim_{\a\to 0} \a ||f||_{b^\a_p}^p \asymp ||f||_p^p.
\end{equation}

We note that in the works \cite{BBM2} and \cite{MS} a slightly
different definition of the  seminorm $||\cdot||_{b^\a_p}$ was used;
it is equivalent to the one given above.

Later on, it was observed in \cite{KL} that inequalities
(\ref{bourg}) and (\ref{mazya}) can be directly derived from
estimates of rearrangements obtained in \cite{K1988}.

Different  extensions and some close aspects of these problems have been studied in \cite{CGO}, \cite{KMX},
\cite{KL}, \cite{LS}, \cite{Mil}, \cite{Tri}.

This paper was inspired by the results described above. Namely, it
is devoted to the   study of limiting behaviour of capacities in
spaces
 $B_{p,q}^\a$ as $\a$ tends to 1 or $\a$ tends to 0.

Let $K\subset \R^n$ be a compact set. Denote by $\mathfrak{N}(K)$ the
set of all functions $f\in C_0^\infty(\R^n)$ such that  $f(x)\ge  1$ for all $x\in K.$
The capacity of the set $K$ in the space $W_p^1(\R^n)\,\, (1\le
p<\infty)$ is defined by
\begin{equation}\label{sobolcap}
\capa(K;W_p^1)=\inf\left\{\left(\sum_{k=1}^n||D_k f||_p\right)^p: f\in \mathfrak{N}(K)\right\}
\end{equation}
(see \cite[2.2.1]{Maz2}).

Similarly, let $1\le p,q<\infty$ and $0<\a<1.$ The capacity of a
compact set $K\subset \R^n$ in the space $B_{p,q}^\a(\R^n)$ is
defined by
\begin{equation}\label{besovcap}
\capa(K;B_{p,q}^\a)=\inf\{||f||_{b_{p,q}^\a}^p: f\in \mathfrak{N}(K)\}
\end{equation}
(see \cite{Ad1}, \cite[Section 4]{AH}, \cite[Section 10.4]{Maz2}).
Note that in this definition the $p$th power of the Besov norm is
taken. This assures that the Hausdorff dimension of the set function
$\capa(\cdot; B_{p,q}^\a)$ is equal to $n-\a p$ when $p<n/\a$ (see
\cite{Ad1}).

Let $X$ denote one of the spaces $W_p^1(\R^n)$ or
$B_{p,q}^\a(\R^n).$ Let $G\subset \R^n$ be an open set. Then we
define the capacity of $G$ in $X$ as
$$
\capa(G;X)=\sup\{\capa(K;X): K\subset G, K \,\,\mbox{is compact}\}.
$$

The paper is organized as follows.

In Section 2 we give auxiliary statements which are used in the
sequel.

In Section 3 we prove the  main result of the paper. It states that
if $1\le p<n$ and $1\le q<\infty$, then for any open set $G\subset
\R^n$
\begin{equation}\label{ato1}
\lim_{\a\to 1-}(1-\a)^{p/q} \capa\left(G; B_{p,q}^\a\right)=
\left(\frac1q\right)^{p/q}\capa (G; W_p^1).
\end{equation}
We show  that this statement may fail for a {\it compact} set. If
$n<p<\infty$, $n\in\N,$ or $n=p\ge 2,$ then equality (\ref{ato1}) is
trivially true because in these cases both the sides of
(\ref{ato1}) are equal to zero for any bounded open set $G.$
Furthermore, (\ref{ato1}) also trivially holds for $p=n=1$; in this case both
the sides are equal to $2q^{-1/q}$ for any non-empty open bounded
set $G\subset \R.$

In Section 4 we consider the case $\a\to 0$ and we prove that if
$1\le p, q<\infty$, then for any compact set $K\subset \R^n$
$$
\lim_{\a\to 0+}\a^{p/q}{\capa} (K;B_{p,q}^\a)=2n^p\left(\frac1q\right)^{p/q}|K|
$$
(as usual, $|K|$ denotes the Lebesgue measure of $K$). It is shown that generally this equality is not true for {\it open}
sets.

\section{Auxiliary propositions}

We begin with some  properties of  moduli of continuity.

We shall call {\it modulus of continuity} any non-decreasing,
continuous and bounded function $\o(\d)$ on $[0,+\infty)$ which
satisfies the conditions
\begin{equation}\label{d1}
 \o(\d+\eta)\le \o(\d)+\o(\eta),
\quad \o(0)=0.
\end{equation}

It is well known that for any $f\in L^p({\mathbb R}^n)$ the functions
$\o_j(f;\d)_p$ are moduli of continuity.

For a modulus of continuity $\o$ the function $\o(\d)/\d$ may not be
monotone. Therefore we shall use the following lemma.

\begin{lem}\label{MOD1} Let $\o$ be a modulus of continuity. Set
$$
\overline\o(t)=\frac1t\int_0^t\o(u)\,du,\quad t>0.
$$
Then
\begin{equation}\label{mod3}
\overline\o(t)\le\o(t)\le 2\overline\o(t),\quad t>0.
\end{equation}
Moreover, $\overline\o(t)$ increases and $\overline\o(t)/t$
decreases on $(0,\infty).$
\end{lem}
\begin{proof} Since
$$
\overline\o(t)=\int_0^1 \o(tv)\,dv
$$
and $\o$ is increasing, it is obvious that $\overline\o$ increases
and the left-hand side inequality in (\ref{mod3}) is true. We prove
the right-hand side inequality in (\ref{mod3}), that is,
\begin{equation}\label{mod4}
\o(t)\le \frac2t\int_0^t\o(u)\,du,\quad t>0.
\end{equation}
We have
$$
\int_0^t\o(u)\,du=\int_0^t\o(t-u)\,du.
$$
Thus, by (\ref{d1}),
$$
2\int_0^t\o(u)\,du=\int_0^t[\o(u)+\o(t-u)]\,du\ge t\o(t).
$$
This implies (\ref{mod4}). Using (\ref{mod4}), we obtain
$$
\left(\frac{\overline\o(t)}{t}\right)'=-\frac{2}{t^3}\int_0^t\o(u)\,du+\frac{\o(t)}{t^2}\le 0.
$$
for almost all $t>0.$ Since $\overline{\o}(t)$ is locally absolutely continuous on $(0,+\infty)$, this implies that $\overline\o(t)/t$ decreases on $(0,+\infty).$
\end{proof}

Now we consider some estimates of partial moduli of continuity.

First, it is obvious that for any $f\in L^p(\R^n)$
$(1\le p<\infty)$
\begin{equation}\label{First}
\o_j(f;\d)_p\le 2||f||_p \quad (j=1,...,n).
\end{equation}
It is easy to show that the constant 2 at the right-hand side is optimal (see Remark 4.3 below). However,
for non-negative functions the constant can be improved. Namely, if $f\in L^p(\R^n)$ and $f(x)\ge 0$, then
\begin{equation}\label{posit}
\o_j(f;\d)_p\le 2^{1/p}||f||_p \quad (j=1,...,n).
\end{equation}
Indeed, let  $h>0$, $j\in\{1,...,n\}$,  and set $E_{h,j}=\{x:
f(x)\ge f(x+he_j)\}$. Then
$$
\begin{aligned}
&\int_{\R^n}|f(x)-f(x+he_j)|^p\,dx\\
&\le \int_{E_{h,j}}f(x)^p\,dx+\int_{\R^n\setminus E_{h,j}}f(x+he_j)^p\,dx\le 2\int_{\R^n}f(x)^p\,dx.
\end{aligned}
$$
This implies (\ref{posit}).

In what follows, for a  set $E\subset \R^n$ we denote by $\chi_E$
its characteristic function. If $E$ is a measurable set of finite
measure, then by (\ref{posit})
\begin{equation}\label{charact}
\o_j(\chi_E;\d)_p\le (2|E|)^{1/p}.
\end{equation}

If a function  $f\in L^p(\R^n)~~(1\le p<\infty)$ has a weak
derivative $D_j f\in L^p(\R^n)$ for some $1\le j\le n,$ then
\begin{equation}\label{ocenka}
\o_j(f;\d)_p\le ||D_j f||_p\d
\end{equation}
(see \cite[\S\, 16]{BIN}). Moreover, by the Hardy-Littlewood theorem
\cite[\S \,4.8]{Nik}, if  $1<p<\infty$ and $f\in L^p(\R^n),$ then
the relation $\o_j(f;\d)_p=O(\d)$ holds if and only if there exists
the weak derivative $D_jf\in L^p(\R^n).$

We shall   also use the
following  well-known statement which we prove for completeness.
\begin{lem}\label{NABLA} Let a function $f\in L^p(\R^n)~~(1\le p<\infty)$
have a weak derivative $D_jf\in L^1_{loc}\R^n$ for some  $j\in\{1,...,n\}.$ Then
\begin{equation}\label{supremum}
||D_j f||_p=\lim_{\d\to 0+}\frac{\o_j(f;\d)_p}{\d}=\sup_{\d>0}\frac{\o_j(f;\d)_p}{\d}.
\end{equation}
\end{lem}
\begin{proof}
The function $f$ can be modified on a set of measure zero so that the modified function is locally absolutely continuous on almost all straight lines parallel to the $x_j-$axis,
and its usual derivative with respect to $x_j$ coincides almost everywhere on $\R^n$ with $D_jf$ (see \cite[Chapter 4]{Nik}). We assume that $f$ itself has this property. Then
$$
\frac{f(x+he_j)-f(x)}{h}\to D_j f(x)\quad \mbox{as}\quad h\to 0
$$
almost everywhere on $\R^n$. Thus, by Fatou's Lemma,
$$
\begin{aligned}
&\left(\int_{\R^n}|D_j f(x)|^p\,dx\right)^{1/p}\\
&\le \varliminf_{h\to 0+}\left(h^{-p}\int_{\R^n}|f(x+he_j)-f(x)|^p\,dx\right)^{1/p}
\le \varliminf_{h\to 0+} \frac{\o_j(f;h)_p}{h}.
\end{aligned}
$$
On the other hand, by (\ref{ocenka})
$$
||D_j f||_p\ge \sup_{h>0}\frac{\o_j(f;h)_p}{h}\ge \varlimsup_{h\to 0+} \frac{\o_j(f;h)_p}{h}.
$$
 These inequalities yield  (\ref{supremum}).
\end{proof}
\begin{rem}\label{REMA} As we have observed above, for a modulus of continuity $\o$ the
function $\o(\d)/\d$ may not be monotone. However, it is not difficult to show that for
any modulus of continuity $\o$
$$
\lim_{\d\to 0+}\frac{\o(\d)}{\d}=\sup_{\d>0}\frac{\o(\d)}{\d}.
$$
\end{rem}

Now we derive some estimates involving Besov norms. First, we have the following  lemma which we shall often use in the sequel.
\begin{lem}\label{partition} Assume that a function $f\in L^p(\R^n)$ ($1\le p<\infty)$ has a weak derivative $D_jf\in L^p(\R^n)$ for some  $j\in \{1,...,n\}.$
Then $f\in B_{p,q;j}^\a(\R^n)$ for any $1\le q<\infty$ and any $0<\a<1.$ Moreover,
$$
\|f\|_{b^\a_{p,q;j}}
\le q^{-1/q}\left[(1-\a)^{-1/q}T^{1-\a}||D_jf||_p+2\a^{-1/q}T^{-\a }||f||_p\right]
$$
for any $T>0$.
\end{lem}
\begin{proof} Applying estimates (\ref{First}) and (\ref{ocenka}), we obtain for $T>0$
$$
\begin{aligned}
\|f\|_{b^\a_{p,q;j}}&\le \left(\int_0^T t^{-\a q}\o_j(f;t)_p^q\frac{dt}{t}\right)^{1/q}
+\left(\int_T^\infty t^{-\a q}\o_j(f;t)_p^q\frac{dt}{t}\right)^{1/q}\\
&\le ||D_j f||_p\left(\int_0^T t^{(1-\a)q}\frac{dt}{t}\right)^{1/q}
+2||f||_p\left(\int_1^\infty t^{-\a q}\frac{dt}{t}\right)^{1/q}\\
&=q^{-1/q}(1-\a)^{-1/q}T^{1-\a}||D_jf||_p+2(\a q)^{-1/q}T^{-\a}||f||_p.
\end{aligned}
$$
\end{proof}

It is  well known that for fixed $\a\in (0,1)$ and $p\in [1,
\infty)$ the Besov spaces $B^\a_{p,q}(\R^n)$ increase as the second
index $q$ increases. Moreover,  the following estimate holds: if
$1\le p<\infty,$ $1\le q<\t\le\infty,$ and $0<\a<1,$ then for any
function $f\in L^p({\mathbb R^n})$ and any $j=1,...,n$
\begin{equation}\label{besov}
\|f\|_{b^\a_{p,\t;j}}\le
8[\a(1-\a)]^{1/q-1/\t}\|f\|_{b^\a_{p,q;j}}
\end{equation}
(see \cite[Lemma 2.2]{K2006}).
The constant coefficient at the right-hand side has optimal order as $\a\to 1$ or $\a\to 0.$
However, the value of this coefficient can be improved. First, for "small" $\a$ we have the following result.
\begin{lem}\label{BESOV1} Let $1\le p<\infty,$ $1\le q<\t\le\infty,$ and $0<\a<1.$ Then
for any function $f\in L^p({\mathbb R^n})$ and any $j=1,...,n$
\begin{equation}\label{besov11}
\|f\|_{b^\a_{p,\t;j}}\le
(\a q)^{1/q-1/\t}\|f\|_{b^\a_{p,q;j}}.
\end{equation}
\end{lem}
\begin{proof} Indeed, for any $\d>0$ and any $j\in\{1,...,n\}$,
$$
\begin{aligned}
&\a||f||_{b^\a_{p,q;j}}^q \ge \a\int_\d^\infty t^{-\a q}\o_j(f;t)_p^q\,\frac{dt}{t}\\
&\ge \o_j(f;\d)_p^q\, \a \int_\d^\infty t^{-\a q}\,\frac{dt}{t}=\frac{1}{q}\d^{-\a q}\o_j(f;\d)_p^q.
\end{aligned}
$$
Thus, we obtain (\ref{besov11}) for $\t=\infty.$  From here, for any
 $\t\in (q,\infty),$ we get
$$
\begin{aligned}
&\|f\|_{b^\a_{p,\t;j}}^\t = \int_0^\infty t^{-\a
\theta}\o_j(f;t)_p^\theta\,\frac{dt}{t}\\
&\le
\|f\|_{b^\a_{p,\infty;j}}^{\t-q}\int_0^\infty t^{-\a
q}\o_j(f;t)_p^q\,\frac{dt}{t}
 \le
(\a q)^{(\t-q)/q}\|f\|_{b^\a_{p,q;j}}^\t.
\end{aligned}
$$
This yields (\ref{besov11}).
\end{proof}

The following lemma plays an essential role in the case $\a\to 1-0.$
\begin{lem}\label{BESOV2} Let $1\le p<\infty,$  $1\le q<\t\le\infty,$ and $0<\a<1.$ Then
for any function $f\in L^p({\mathbb R^n})$ and any $j=1,...,n$
\begin{equation}\label{besov21}
\|f\|_{b^\a_{p,\t;j}}\le
[(1-\a)q]^{1/q-1/\t}\left(\frac{2}{1+\a}\right)^{1-q/\t}\|f\|_{b^\a_{p,q;j}}.
\end{equation}
\end{lem}
\begin{proof} Fix $j\in\{1,...,n\}$ and
 set
$$
\overline \o(t)=\frac1t\int_0^t\o_j(f;u)_p\,du,\quad t>0.
$$
By Hardy's inequality \cite[p. 124]{BS},
$$
\int_0^\infty t^{-\a
q}\overline\o(t)^q\frac{dt}{t}\le\frac1{(1+\a)^q}\int_0^\infty t^{-\a
q}\o_j(f;t)_p^q\frac{dt}{t}.
$$
Using this estimate, we have
$$
\begin{aligned}
\|f\|_{b^\a_{p,q;j}}^q&=\int_0^\infty t^{-\a q}\o_j(f;t)_p^q\frac{dt}{t}\\
&\ge (1+\a)^q\int_0^\infty t^{-\a
q}\overline\o(t)^q\frac{dt}{t}\\
&\ge (1+\a)^q\int_0^\d t^{(1-\a)q}\left(\frac{\overline\o(t)}{t}\right)^q\frac{dt}{t}
\end{aligned}
$$
for any $\d>0.$ By Lemma \ref{MOD1}, $\overline\o(t)/t$ decreases on
$(0,+\infty).$ Hence,
$$
\begin{aligned}
 (1-\a)\|f\|_{b^\a_{p,q;j}}^q &\ge (1+\a)^q(1-\a)\left(\frac{\overline\o(\d)}{\d}\right)^q\int_0^\d t^{(1-\a)q}\frac{dt}{t}\\
&=\frac{(1+\a)^q}{q}\d^{-\a q}\overline\o(\d)^q, \quad\d>0.
\end{aligned}
$$
By (\ref{mod3}), $\o_j(f;\d)_p\le 2\overline\o(\d),$ and thus we
obtain
$$
(1-\a)\|f\|_{b^\a_{p,q;j}}^q\ge\frac1q\left(\frac{1+\a}{2}\right)^q\d^{-\a
q}\o_j(f;\d)_p^q, \quad\d>0.
$$
This implies inequality (\ref{besov21}) for $\t=\infty.$ In the case $\t<\infty$ this inequality follows as in the proof of Lemma \ref{BESOV1}.
\end{proof}

Next, we consider  some estimates of {\it distribution functions.}

For any measurable function  $f$ on $\mathbb{R}^n$, denote
$$
\lambda_f (y) = |\{x \in \mathbb{R}^n : |f(x)|>y \}|, \quad y>0.
$$
Let $S_0(\mathbb{R}^n)$ be the class of all measurable and almost
everywhere finite functions $f$ on $\mathbb{R}^n$ such that
$\lambda_f (y)< \infty$  for each $y>0$.

{\it A non-increasing rearrangement} of a function $f \in
S_0(\mathbb{R}^n)$ is a non-increasing function $f^*$ on $(0, +
\infty)$ such that for any $y>0$
$$
|\{t>0: f^*(t)>y\}|= \lambda_f (y).
$$
We shall assume in addition that the rearrangement $f^*$ is left
continuous on $(0,\infty).$ Under this condition it is defined
uniquely by
$$
f^*(t)=\inf\{y>0: \lambda_f (y)<t\}, \quad 0<t<\infty.
$$
It follows that
\begin{equation}\label{distrib}
f^*(\l_f(y))\ge y \quad\mbox{for any} \quad y\ge 0.
\end{equation}

Set also
$$
f^{**}(t)=\frac1t\int_0^t f^*(u)\,du.
$$
For any $f\in S_0(\R^n)$
\begin{equation}\label{deriv}
f^{**}(t)=\int_t^\infty \frac{f^{**}(u)-f^*(u)}{u}\,du,\quad t>0.
\end{equation}

If $f\in S_0(\R^n)$ is locally integrable and has all weak
derivatives $D_k f \in L^1_{\loc}~~(k=1,...,n),$
then
\begin{equation}\label{estim1}
f^{**}(t) - f^*(t)\le n~t^{1/n}\sum_{k=1}^n(D_k f)^{**}(t) \quad (t>0)
\end{equation}
(see \cite[Lemma 5.1]{K1989}, \cite[Lemma 3.1]{K2007}).

\begin{lem}\label{MEASURE1} Let $f\in W_p^1(\R^n),\,\, 1\le p<n$,
and let $p^*=np/(n-p).$ Then
\begin{equation}\label{measure1}
\lambda_f(y)\le c_{p,n}\left(\sum_{k=1}^n||D_k f||_p\right)^{p^*}y^{-p^*}, \quad y>0.
\end{equation}
\end{lem}
\begin{proof}
Of course, this weak-type inequality follows from the strong-type
inequality (\ref{sobolev}). However, (\ref{measure1}) is a direct
consequence of the estimate (\ref{estim1}). Indeed, by (\ref{deriv})
and (\ref{estim1}),
$$
\begin{aligned}
&f^*(t) \le f^{**}(t)\le n\sum_{k=1}^n\int_t^\infty
u^{1/n-1}(D_k f)^{**}(u)\,du\\
&=nn'\sum_{k=1}^n\left[t^{1/n-1}\int_0^t (D_k f)^*(u)\,du + \int_t^\infty
u^{1/n-1}(D_k f)^*(u)\,du\right].
\end{aligned}
$$
 Applying H\"older inequality to  both the
integrals at the right-hand side, we have
$$
f^*(t)\le c\, t^{-1/p^*} \sum_{k=1}^n||D_k f||_p.
$$
Setting $t=\l_f(y)$ and taking into account (\ref{distrib}), we get
(\ref{measure1}).
\end{proof}

Similarly, estimates of distribution functions in terms of moduli of
continuity can be derived from the following inequality: for any $f\in L^p(\R^n)\,\,\,(1\le p<\infty)$
\begin{equation}\label{estim2}
f^{**}(t) - f^*(t)\le 2 t^{-1/p}\sum_{k=1}^n\o_k(f; t^{1/n})_p.
\end{equation}
This  inequality was first proved by  Ul'yanov \cite{U} in the
one-dimensional case (see \cite[p. 148]{K1998} for an alternative
proof). For all $n\ge 1$ it was proved in \cite{K1975}; a simpler proof was given in
\cite[Theorem 1]{K1988}.

\begin{lem}\label{MEASURE2} Let $0<\a<1,$  $1\le p< n/\a,$ and
$p_\a=np/(n-\a p).$ Then for any function $f\in L^p(\R^n)$
\begin{equation}\label{measure2}
\lambda_f(y)\le (2p_\a)^{p_\a}||f||_{b_{p,\infty}^\a}^{p_\a}y^{-p_\a}, \quad y>0.
\end{equation}
\end{lem}
\begin{proof} We have
$$
\sum_{k=1}^n\o_k(f;t)_p\le t^\a ||f||_{b_{p,\infty}^\a}\quad\mbox{for any}\quad t\ge 0.
$$
 Thus, by (\ref{deriv}) and (\ref{estim2}),
$$
\begin{aligned}
&f^*(t)\le f^{**}(t)\le 2\sum_{k=1}^n\int_t^\infty u^{-1/p}\o_k(f;u^{1/n})_p\,\frac{du}{u}\\
&\le 2||f||_{b_{p,\infty}^\a}\int_t^\infty u^{-1/p+\a/n}\,\frac{du}{u}=2p_\a||f||_{b_{p,\infty}^\a} t^{-1/p_\a}.
\end{aligned}
$$
Setting  $t=\l_f(y)$ and applying (\ref{distrib}), we obtain
(\ref{measure2}).
\end{proof}

We shall use the following notations. For any $x=(x_1,...,x_n)\in
\R^n$ we denote by $\widehat{x}_k$ the $(n-1)-$dimensional vector
obtained from the $n$-tuple $x$ by removal of its $k$th coordinate.
Let $E\subset \R^n.$ For every $k=1,...,n,$ denote by $\Pi_k(E)$ the
orthogonal projection of $E$ onto the coordinate hyperplane $x_k=0.$
If $E$ is a set of the type $F_\s$, then all its projections $\Pi_k(E)$ are
sets of the type $F_\s$ in $\R^{n-1}$ and therefore they are
measurable in $\R^{n-1}$. The $(n-1)-$dimensional measure of the
projection $\Pi_k(E)$ will be denoted by $\mes_{n-1}\Pi_k(E)$. For the
$n-$dimensional measure of the set $E$ we keep the usual notation
$|E|.$ As above, by $e_k$ we denote the $k$th unit coordinate vector.

\begin{lem}\label{TRANSLATION} Let $\mu, \,\l,$ and $\eta$ be positive numbers and let $n\in \N.$
Then for any set $E\subset \R^n$ of the type $F_\s$, satisfying the
conditions
\begin{equation}\label{transl1}
|E|\le \mu\quad \mbox{and}\quad \mes_{n-1}\Pi_k (E)\ge \l\quad (k=1,...,n),
\end{equation}
there exists $0<h\le 2\mu^2n/(\l \eta)$ such that
\begin{equation}\label{translation}
\sum_{k=1}^n|\{x\in E: x+he_k\in E\}|<\eta.
\end{equation}
\end{lem}
\begin{proof} Let $E\subset \R^n$  satisfy (\ref{transl1}). Denote
$$
\f_{E,k}(h)=|\{x\in E: x+he_k\in E\}|=\int_E\chi_E(x+he_k)\,dx\quad (h>0).
$$
For any $H>0$ and any $k=1,...,n$,  we have
$$
\begin{aligned}
\int_0^H\f_{E,k}(h)\,dh
&=\int_E dx\int_0^H\chi_E(x+he_k)\,dh\\
&\le |E|\int_\R \chi_E(y)\,dy_k.
\end{aligned}
$$
Integrating over projection $\Pi_k (E),$ we obtain
$$
\begin{aligned}
&\mes_{n-1}\Pi_k (E)\int_0^H\f_{E,k}(h)\,dh\\
&\le |E|\int_{\Pi_k E}d\widehat y_k\int_\R\chi_E(y)\,dy_k = |E|^2.
\end{aligned}
$$
 By (\ref{transl1}), this implies that
$$
\int_0^H\f_{E,k}(h)\le \frac{\mu^2}{\l}\quad (k=1,...,n).
$$
Denoting
$$
\f_E(h)=\sum_{k=1}^n\f_{E,k}(h),
$$
we have
$$
\int_0^H\f_E(h)\,dh\le \frac{\mu^2n}{\l}.
$$
Thus,
$$
\inf_{h\in [0,H]}\f_E(h)\le \frac{\mu^2n}{\l H}
$$
Setting $H=(2\mu^2n/(\l \eta),$ we obtain that there exists
$h\in(0,H]$ (depending on $\mu,\,\l,\,\eta,$ and $E$) such that
$\f_E(h)<\eta.$
\end{proof}

Throughout this paper $\mathcal {B}_r$ denotes the open ball with
radius $r>0$ centered at the origin. In the sequel  we shall use
{\it the standard mollifier} (see, e.g, \cite[p. 553]{Leo1})
\begin{equation}\label{moll}
 \f(x)=\begin{cases}
    c \operatorname{exp}(1/(|x|^2-1)) &\textnormal{if }\quad x\in \mathcal {B}_1 \\
    0 &\textnormal{if}\quad x\not\in \mathcal {B}_1,
  \end{cases}
\end{equation}
where $c>0$ is such that
$$
\int_{\R^n} \f(x)\,dx=1.
$$
Set for $\tau >0$
\begin{equation}\label{moll2}
 \f_\tau(x)=\frac1{\tau^n}\f\left(\frac{x}{\tau}\right).
 \end{equation}
 Then $\f_\tau(x)=0$ if $|x|>\tau,$ and
 \begin{equation}\label{moll3}
 \int_{\R^n} \f_\tau(x)\,dx=1.
 \end{equation}

 We shall  also use the following {\it cutoff function}
 \begin{equation}\label{cutoff}
 \eta(x)= (\f\ast g)(x),
 \end{equation}
 where $g$ is the characteristic function of the open ball $\mathcal {B}_2$. We have that $\eta\in C_0^\infty,$
 $\eta(x)=1$ if $|x|\le 1$ and $\eta(x)=0$ if $|x|\ge 3$.

 Let $f\in C^\infty(\R^n)\cap W_p^1(\R^n).$ For any $\g>0$ the function $f_\gamma(x)=f(x)\eta(\g x)$
 belongs to $C_0^\infty(\R^n)$. Moreover, it is easy to see that for
 any $\e>0$ there exists  $\gamma_0>0$ such that for all $0<\g\le
 \g_0$
 \begin{equation}\label{gamma}
 ||D_k f_\g||_p<||D_k f||_p + \e \quad (k=1,...,n)
\end{equation}
(see, e.g., \cite[p. 124]{St}).

In the sequel we use also the following remark concerning
capacities. Let $K\subset \R^n$ be a compact set. Denote by
$\mathfrak{P}(K)$ the set of all functions $f\in C_0^\infty(\R^n)$
such that $0\le f(x)\le 1$ for all $x\in \R^n$ and $f(x)=1$ in some
neighborhood of $K.$ It is well known that the set $\mathfrak{N}(K)$
in definitions (\ref{sobolcap}) and (\ref{besovcap}) may be replaced
by $\mathfrak{P}(K)$. Namely,
$$
\capa(K;W_p^1)=\inf\left\{\left(\sum_{k=1}^n||D_k f||_p\right)^p: f\in \mathfrak{P}(K)\right\}
$$
and
$$
\capa(K;B_{p,q}^\a)=\inf\{||f||_{b_{p,q}^\a}^p: f\in \mathfrak{P}(K)\}
$$
(see \cite[2.2.1]{Maz2}).

 \vskip 6pt
\section{The limit as $\a\to 1$}

In this section we prove the main result of the paper. As we have
already mentioned in the Introduction, this result was inspired by
the limiting relation (\ref{limiting1}) proved in \cite{BBM1}. We
observe that  the following slight modification of (\ref{limiting1})
holds: if  a function $f\in L^p({\mathbb R}^n)$ has a weak
derivative $D_jf \in L^p({\mathbb R}^n)$, then  $f\in
B_{p,q;j}^\a(\R^n)$ for any $1\le q <\infty$ and any $0<\a<1$, and
$$
\lim_{\a\to
1-0}(1-\a)^{1/q}\|f\|_{b^{\a}_{p,q;j}}=\left(\frac{1}{q}\right)^{1/q}
\|D_jf\|_p\,.
$$
This statement follows by standard arguments from Lemma \ref{NABLA} and inequality (\ref{First}) (see also \cite[Section 14.3]{Leo1}).

\begin{teo}\label{MAIN} Let $n\ge 2$,  $1\le p<n,$ and $1\le
q<\infty$. Then for any open set $G\subset \R^n$
\begin{equation}\label{main1}
\lim_{\a\to 1-0}(1-\a)^{p/q} \capa\left(G; B_{p,q}^\a\right)=
\left(\frac1q\right)^{p/q}\capa (G; W_p^1).
\end{equation}
\end{teo}
\begin{proof} Denote
\begin{equation}\label{main0}
\Lambda(\a)=(1-\a)^{1/q}\left[\capa(G; B_{p,q}^\a\right)]^{1/p},\quad 0<\a<1.
\end{equation}
First we shall show that
\begin{equation}\label{main2}
\varlimsup_{\a\to 1-0}\Lambda(\a)\le q^{-1/q}[\capa(G;W_p^1)]^{1/p}.
\end{equation}
We assume that $\capa(G;W_p^1)<\infty.$ Let $K\subset G$ be a
compact set and let $0<\e<1.$ There exists a function $f\in
C_0^\infty(\R^n)$ such that
\begin{equation}\label{main3}
\sum_{k=1}^n||D_k f||_p < (\capa(K;W_p^1)+\e)^{1/p},
\end{equation}
$0\le f(x)\le 1$ for all $x\in \R^n,$ and $f(x)=1$ in some
neighborhood of $K.$ Set $E_\e=\{x:f(x)>\e\}.$ By Lemma
\ref{MEASURE1},
$$
|E_\e|\le c_{p,n}\left(\sum_{k=1}^n||D_k f||_p\right)^{p^*}\e^{-p^*}, \quad p^*=\frac{np}{n-p}.
$$
Using (\ref{main3}) and taking into account that $K\subset G,$ we
obtain that
\begin{equation}\label{main4}
|E_\e|\le A \e^{-p^*},
\end{equation}
where $A\equiv A(n,p,G)=c_{p,n}[\capa(G;W_p^1))+1]^{p^*/p}$. We emphasize that $A$
doesn't depend on $K.$

There exists an open set  $H$  such that $K\subset H$ and $f(x)=1$
on $H$. Let $\rho$ be the distance from $K$ to the boundary of $H$
and let $0<\tau<\rho/2.$ Set
$$
f_\e(x)=\frac1{1-\e}\max(f(x)-\e, 0)\quad\mbox{and}\quad f_{\e,\tau}(x)= (f_\e\ast \f_\tau)(x),
$$
where $\f_\tau$ is defined by (\ref{moll2}). Then $f_\e\in
W_p^1(\R^n)$ and
$$
||D_k f_\e||_p\le \frac1{1-\e}||D_k f||_p \quad (k=1,...,n).
$$
 Furthermore, $D_k f_{\e,\tau}$=$(D_k f_\e)\ast\f_\tau.$ Thus, by (\ref{moll3}) and Young inequality,
\begin{equation}\label{main5}
||D_k f_{\e,\tau}||_p\le ||D_k f_\e||_p\le \frac1{1-\e}||D_k f||_p \quad (k=1,...,n).
\end{equation}
It is clear that $f_\e(x)=0$ if $x\not\in E_\e$ and $0\le f_\e(x)\le 1$ for all $x\in \R^n.$ First, by
 (\ref{moll3}) and (\ref{main4}), this imply that
\begin{equation}\label{main7}
||f_{\e,\tau}||_p\le||f_\e||_p\le |E_\e|^{1/p} \le (A\e^{-p^*})^{1/p},\quad A=A(n,p,G).
\end{equation}
We have also that $0\le f_{\e,\tau}(x)\le 1$ for all $x\in \R^n.$ Furthermore, $f_\e(x)=1$
on $H.$ This yields that
$f_{\e,\tau}(x)=1$ for all $x$ such that $\dist(x,K)<\tau.$
Indeed, if $\dist(x,K)<\tau$ and $|y|\le \tau$, then $x-y\in H$ and $f_\e(x-y)=1.$
Thus,
$$
f_{\e,\tau}(x)=\int_{B_\tau}\f_\tau(y) f_\e(x-y)\,dy=1.
$$
Observe also that
$f_{\e,\tau}\in C_0^\infty(\R^n).$ Taking into account these properties of $f_{\e,\tau}$, we have that
\begin{equation}\label{main6}
\capa\left(K;B_{p,q}^\a\right)\le ||f_{\e,\tau}||_{b_{p,q}^\a}^p.
\end{equation}
Applying Lemma \ref{partition} with $T=1$, we obtain
$$
(1-\a)^{1/q}||f_{\e,\tau}||_{b_{p,q}^\a} \le  \left(\frac1q\right)^{1/q}\left[\sum_{k=1}^n ||D_k f_{\e,\tau}||_p+2\left(\frac{1-\a}{\a}\right)^{1/q}||f_{\e,\tau}||_p\right].
$$
Using  (\ref{main5}) and (\ref{main3}) and taking into account that
$K\subset G,$ we have
$$
\sum_{k=1}^n||D_k f_{\e,\tau}||_p\le \frac1{1-\e}\left[\capa (G;W_p^1)+\e\right]^{1/p}.
$$
The last two inequalities, together with (\ref{main7}) and (\ref{main6}), yield that
$$
\begin{aligned}
&(1-\a)^{1/q}[\capa\left(K;B_{p,q}^\a\right)]^{1/p}\\
&\le \frac1{q^{1/q}(1-\e)}[\capa(G;W_p^1)+\e]^{1/p}+ A'\left(\frac{1-\a}{\a q}\right)^{1/q}\e^{-p^*/p},
\end{aligned}
$$
where $A'=2A(n,p,G)^{1/p}$ doesn't depend on $K.$ Taking
supremum over all compact sets $K\subset G$ and using notation
(\ref{main0}), we get
$$
\Lambda(\a)\le\frac{1}{q^{1/q}(1-\e)}[\capa(G;W_p^1)+\e]^{1/p}+ A'\left(\frac{1-\a}{\a q}\right)^{1/q}\e^{-p^*/p}.
$$
It follows that
$$
\varlimsup_{\a\to 1-0}\Lambda(\a)\le \frac{1}{q^{1/q}(1-\e)}[\capa(G;W_p^1)+\e]^{1/p}.
$$
Since $\e\in(0,1)$ is arbitrary, this implies (\ref{main2}).

Now we shall prove that
\begin{equation}\label{main8}
\varliminf_{\a\to 1-0}\Lambda(\a)\ge q^{-1/q}[\capa(G;W_p^1)]^{1/p}.
\end{equation}
Let $K\subset G$ be a compact set. Choose $\tau>0$ such that
\begin{equation}\label{main9}
K_\tau=\{x\in \R^n: \dist(x, K)\le 2\tau\}\subset G.
\end{equation}
Then $K_\tau$ is compact.

We assume that $\varliminf_{\a\to 1-0}\Lambda(\a)<\infty.$ There
exists an increasing sequence $\{\a_\nu\}$ of numbers $\a_\nu\in
(0,1)$ such that $\a_\nu\to 1$ and
\begin{equation}\label{main10}
\lim_{\nu\to\infty} \Lambda(\a_\nu)=\varliminf_{\a\to 1-0}\Lambda(\a).
\end{equation}
We assume also that
\begin{equation}\label{main11}
\Lambda(\a_\nu)\le \varliminf_{\a\to 1-0}\Lambda(\a)+1 \quad (\nu\in \N).
\end{equation}
 For any $\nu\in \N$ there
exists a function $f_\nu\in C_0^\infty(\R^n)$ such that $0\le
f_\nu(x)\le 1$ for all $x\in\R^n$, $f_\nu(x)=1$ for all $x\in
K_\tau$, and
$$
\left\|f_\nu\right\|_{b_{p,q}^{\a_\nu}}\le \capa\left(K_\tau;B_{p,q}^\a\right)^{1/p}+\frac1{\nu}.
$$
Since $K_\tau\subset G,$ then
$\capa\left(K_\tau;B_{p,q}^\a\right)\le
\capa\left(G;B_{p,q}^\a\right),$ and we have
\begin{equation}\label{main12}
(1-\a_\nu)^{1/q}\left\|f_\nu\right\|_{b_{p,q}^{\a_\nu}}\le \Lambda(\a_\nu) +\frac1{\nu}.
\end{equation}
We shall estimate $\o_j(f_\nu;\d)_p.$ Using (\ref{main12}) and Lemma
\ref{BESOV2} with $\theta=\infty$, we obtain that
\begin{equation}\label{main14}
\Lambda(\a_\nu)+\frac1\nu\ge q^{-1/q}\frac{1+\a_\nu}{2}\d^{-\a_\nu q}\sum_{j=1}^n\o_j(f_\nu;\d)_p
\end{equation}
for any $\d>0$ and any $\nu\in\N.$ In particular, (\ref{main14}) and
 (\ref{main11})  yield that
\begin{equation}\label{main15}
\sum_{k=1}^n\o_k(f_\nu;\d)_p\le A\d^{\a_\nu}, \quad \d>0,
\end{equation}
where $A=2q^{1/q}(\varliminf_{\a\to 1-0}\Lambda(\a)+2)$ depends only on $p,q,n,$ and $G.$ To get also a control of $L^p-$norms, we apply truncation
to the functions $f_\nu$. Let $0<\e<1/2.$ Set
$$
E_{\nu,\e}=\{x\in\R^n: f_\nu(x)>\e\}.
$$
Let $p^*=np/(n-p)$ and $p_\nu=np/(n-\a_\nu p)$; then $p_\nu<p^*.$ By
Lemma \ref{MEASURE2},
$$
|E_{\nu,\e}|\le (2p_\nu)^{p_\nu}\e^{-p_\nu}||f_\nu||_{b_{p,\infty}^{\a_\nu}}^{p_\nu}\le (2p^*)^{p^*}\e^{-p^*}||f_\nu||_{b_{p,\infty}^{\a_\nu}}^{p_\nu}.
$$
Thus, using (\ref{main15}), we obtain
\begin{equation}\label{main16}
|E_{\nu,\e}|\le A'\e^{-p^*} \quad (\nu\in \N),
\end{equation}
where  $A'$ depends only on $p,q,n,$ and $G,$ $A'=(2p^*A)^{p^*}.$  Set now
$$
f_{\nu,\e}(x)=\frac1{1-\e}\max (f_\nu(x)-\e, 0).
$$
It is easily seen that
\begin{equation}\label{main17}
\o_j(f_{\nu,\e};\d)_p\le \frac1{1-\e}\o_j(f_\nu;\d)_p, \quad \d\ge 0\,\,\,\,(j=1,...,n).
\end{equation}
Moreover, $0\le f_{\nu,\e}(x)\le 1$ for all $x\in\R^n,$
$f_{\nu,\e}(x)=1$ for all $x\in K_\tau,$ and $f_{\nu,\e}(x)=0$ for
$x\not\in E_{\nu,\e}.$ Applying (\ref{main16}), we get
\begin{equation}\label{main18}
||f_{\nu,\e}||_p^p\le |E_{\nu,\e}| \le A'\e^{-p^*} \quad (\nu\in \N).
\end{equation}
Besides, by (\ref{main15}) and (\ref{main17}),
\begin{equation}\label{main19}
\o(f_{\nu,\e};\d)_p\le 2A'\d^{\a_1}, \quad \d\in [0,1], \,\,\nu\in \N.
\end{equation}
By virtue of (\ref{main18}), (\ref{main19}), and the compactness
criterion (see \cite[p. 111]{Br2}), for any compact set $Q\subset\R^n$ there exists a subsequence of
$\{f_{\nu,\e}\}$ that converges in $L^p(Q).$ Therefore, by Riesz's theorem, for any compact set $Q\subset\R^n$ there exists a subsequence of
$\{f_{\nu,\e}\}$ that converges almost everywhere on $Q$. Let $Q_s=[-s,s]^n, \,\,\, s\in \N.$ A successive extraction of  subsequences gives strictly increasing sequences $\{\nu_m^{(s)}\}$ $(s=1,2,...)$
of natural numbers such that
$$
\{\nu_m^{(1)}\}\supset \{\nu_m^{(2)}\}\supset ... \supset \{\nu_m^{(s)}\}\supset ...
$$
and for each $s\in N$ the subsequence $\{f_{\nu_m^{(s)},\e}\}$ converges almost everywhere on $Q_s.$
Then the diagonal subsequence $\{f_{\nu_s^{(s)},\e}\}$ converges almost everywhere on $\R^n.$ For simplicity, we  assume that $\{f_{\nu,\e}\}$ itself
converges almost everywhere on $\R^n.$ Let
$$
f_\e(x)=\lim_{\nu\to \infty} f_{\nu,\e}(x).
$$
 Since $f_{\nu,\e}(x)=1$ on $K_\tau$  for any $\nu\in \N,$ then
\begin{equation}\label{main20}
 f_\e(x)=1 \quad \mbox{for  all}\quad x\in K_\tau.
\end{equation}
We have also that $0\le f_\e(x)\le 1$ almost everywhere on $\R^n.$
Further, by Fatou's lemma and (\ref{main18})
\begin{equation}\label{main1000}
||f_{\e}||_p^p\le A'\e^{-p^*}.
\end{equation}
 Fatou's lemma yields also that
for any $h>0$ and any $j=1,...,n$
$$
\int_{\R^n}|f_{\e}(x+he_j)-f_{\e}(x)|^p\,dx\le \varliminf_{\nu\to \infty}\int_{\R^n}|f_{\nu,\e}(x+he_j)-f_{\nu,\e}(x)|^p\,dx.
$$
Thus,
\begin{equation}\label{main21}
\o_j(f_\e;\d)_p\le \varliminf_{\nu\to \infty}\o_j(f_{\nu,\e};\d)_p, \quad \d\ge 0\,\,\,(j=1,...,n).
\end{equation}

Let  $\f_\tau$ be the mollifier defined by (\ref{moll2}).
 Set
$f_{\e,\tau}=f_\e\ast\f_\tau.$ Clearly, $0\le f_{\e,\tau}(x)\le 1$ for all $x\in \R^n$ and, by (\ref{moll3}) and
(\ref{main20}),
\begin{equation}\label{main22}
f_{\e,\tau}(x)=1 \quad\mbox{if}\quad \dist(x,K)<\tau.
\end{equation}
Besides, by Young inequality and (\ref{moll3}),
\begin{equation}\label{main23}
\o_j(f_{\e,\tau};\d)_p\le \o_j(f_\e;\d)_p, \quad\d\ge 0\,\,\, (j=1,...,n).
\end{equation}
Applying inequalities (\ref{main14}) and (\ref{main17}), we obtain
$$
\Lambda(\a_\nu)+\frac1{\nu}\ge \frac{(1+\a_\nu)(1-\e)}{2q^{1/q}}\d^{-\a_\nu}\sum_{j=1}^n\o_j(f_{\nu,\e};\d)_p.
$$
By (\ref{main10}),
(\ref{main21}), and (\ref{main23}), this implies that
\begin{equation}\label{main24}
\varliminf_{\a\to 1-0}\Lambda(\a)\ge\frac{1-\e}{q^{1/q}}\sum_{j=1}^n\frac{\o_j(f_{\e,\tau};\d)_p}{\d}
\end{equation}
for any $\d>0.$ Taking into account (\ref{main1000}), we have  $f_{\e,\tau}\in L^p(\R^n)\cap C^\infty(\R^n).$ Making $\d$ tend to zero and applying Lemma \ref{NABLA}, we obtain
\begin{equation}\label{main25}
\varliminf_{\a\to 1-0}\Lambda(\a)\ge\frac{1-\e}{q^{1/q}}\sum_{j=1}^n||D_j f_{\e,\tau}||_p.
\end{equation}
Let $\eta$ be the cutoff function defined by (\ref{cutoff})). Set $g(x)=f_{\e,\tau}(x)\eta(\g x),$ $\g>0.$ Then $g\in
C_0^\infty(\R^n)$ and $0\le g(x)\le 1$ for all $x\in \R^n$. If $\gamma$ is sufficiently small, then, by virtue of (\ref{main22}),
$g(x)=1$ if $\dist(x,K) <\tau.$ Moreover, $\g$ can be chosen so small that (see (\ref{gamma}))
$$
||D_j g||_p<||D_j f_{\e,\tau}||_p + \frac{\e}{n} \quad (j=1,...,n).
$$
Since
$$
\sum_{j=1}^n||D_j g||_p\ge \capa(K;W_p^1)^{1/p},
$$
 inequality (\ref{main25}) yields that
$$
\varliminf_{\a\to
1-0}\Lambda(\a)\ge\frac{1-\e}{q^{1/q}}[\capa(K;W_p^1)^{1/p}-\e].
$$
Taking into account that $\e\in (0,1)$ and a compact set $K\subset
G$ are arbitrary, we obtain inequality (\ref{main8}). Together with
(\ref{main2}), this gives (\ref{main1}).
\end{proof}

\begin{rem}\label{REMARK1} The statement of Theorem \ref{MAIN} fails to hold for {\it
compact}
sets. To show it, we use a theorem on capacity of a Cantor set
\cite[Section 5.3]{AH}. Let $1<p<n, \,\,\,p=q,$ and let $0<\a<1$. It
is known that in this case the $B_p^\a-$capacity is equivalent to
the Bessel capacity $C_{\a,p}$ \cite[p. 107]{AH}. Set
$$
l_k=((k+4)^2 2^{-kn})^{1/(n-p)}\quad (k=0,1,...).
$$
Then $l_{k+1}<l_k/2$ for all $k\ge 0.$
Further,
$$
\sum_{k=0}^\infty 2^{-kn}l_k^{p-n}<\infty \quad\mbox{and}\quad \sum_{k=0}^\infty 2^{-kn}l_k^{\a p-n}=\infty
$$
for any $0<\a<1.$ Let $E$ be the  Cantor set corresponding to the sequence $\{l_k\}$ defined in
\cite[(5.3.1)]{AH}. It follows by \cite[Theorem 5.3.2]{AH} that
$$
\capa(E;W_p^1)>0 \quad\mbox{and}\quad \capa(E;B_p^\a)=0
$$
for any $0<\a<1.$ Thus, equality (\ref{main1}) does not hold.
\end{rem}

\begin{rem}\label{REMARK2} We observe that if $n<p<\infty, \,\, n\in\N,$ or $p=n\ge 2$, then equality (\ref{main1}) is
trivially true. It is closely related to the fact that in these
cases the Sobolev capacity of a ball in $\R^n$ is equal to zero (see
\cite[p. 148]{Maz2}). For completeness, we give the corresponding
arguments in detail.

First, let $n<p<\infty.$ We consider the ball $\mathcal B_r$, $r>0.$
Let $\eta$ be the cutoff function defined by (\ref{cutoff}). Set
$f_\g(x)=\eta(\gamma x)$, where $0<\gamma<1/r.$ Then $f_\g\in
C_0^\infty(\R^n)$, $0\le f_\g(x)\le 1$ for all $x\in\R^n$, and
$f_\g(x)=1$ in some neighborhood of  $\overline{\mathcal B}_r.$ Further,
\begin{equation}\label{zamena}
||D_k f_\g||_p=\g^{1-n/p}||D_k \eta||_p\quad (k=1,...,n).
\end{equation}
This implies that $\capa (\mathcal B_r; W_p^1)=0$. Moreover, if
$n/p<\a<1,$ then we have also that
\begin{equation}\label{remark3}
\capa(\mathcal B_r; B_{p,q}^\a)=0
\end{equation}
for any $1\le q<\infty$. Indeed,
$$
\capa(\mathcal B_r; B_{p,q}^\a)\le ||f_\g||_{b_{p,q}^\a}^p.
$$
Thus, applying Lemma \ref{partition} and (\ref{zamena}),
we obtain
$$
\begin{aligned}
&||f_\g||_{b_{p,q;k}^\a}\le  q^{-1/q}\left[(1-\a)^{-1/q}T^{1-\a}||D_kf_\g||_p+2\a^{-1/q}T^{-\a }||f_\g||_p\right] \\
&\le  ((1-\a)q)^{-1/q}||D_k \eta||_p T^{1-\a}\gamma^{1-n/p} + 2(\a q)^{-1/q}||\eta||_p T^{-\a } \gamma^{-n/p}
\end{aligned}
$$
for any $T>0$ and any $1\le k\le n$. Setting $T=1/\gamma,$ we get
$$
||f_\g||_{b_{p,q;k}^\a}\le \left[((1-\a)q)^{-1/q}||D_k \eta||_p + 2(\a q)^{-1/q}||\eta||_p\right]\g^{\a-n/p}.
$$
 Since $0<\g<1/r$ is arbitrary and $\a>n/p$, this implies (\ref{remark3}). Thus, if $p>n$, then for any open set $G\subset \R^n$ both the capacities in relation (\ref{main1}) are equal to 0.

Let now $p=n\ge 2.$ We have $\capa (\mathcal B_r; W_n^1)=0 \,\,\,(r>0)$ \cite[p. 148]{Maz2}).
At the same time, it follows from Lemma \ref{MEASURE2} and
inequality (\ref{besov}) that
 $\capa(\mathcal B_r;B_{n,q}^\a)>0$ for any $0<\a<1$ and any $1\le q<\infty.$
Nevertheless, we shall show that
\begin{equation}\label{pequaln}
\lim_{\a\to 0} (1-\a)^{n/q}\capa(\mathcal B_r; B_{n,q}^\a)=0\quad (r>0).
\end{equation}
Let $\s=(n-1)/(2n)$ and set
$$
f_0(x)=\begin{cases}
    |\ln |x||^\s &\textnormal{if }\quad |x|\le 1 \\
    0 &\textnormal{if}\quad |x|>1.
  \end{cases}
$$
It is easy to see that $f\in W_n^1(\R^n).$ Let $\e>0.$ Set $f_1(x)=\min(\e f_0(x),1).$ Since $f_0(x)\to +\infty$ as $x\to 0,$
there exists a closed ball $U_\e$ centered at the origin such that $f_1(x)=1$ for all $x\in U_\e.$ There is $\g>0$  such that $\g x\in U_\e$
for all $x\in \overline{\mathcal B}_{r+1}.$
Set $f_2(x)=f_1(\g x).$ Then
$$
||D_k f_2||_n=||D_k f_1||_n\le \e ||D_k f_0||_n \quad (k=1,...,n)
$$
and
$$
||f_2||_n=\frac{||f_1||_n}{\g}\le \frac{\e ||f_0||_n}{\g}.
$$
Finally, we define $f=f_2\ast \f_{1/2}$ (see (\ref{moll2})). Then $f\in C_0^\infty(\R^n)$,
$f(x)=1$ in $\mathcal B_{r+1/2},$ and $0\le f(x)\le 1$ for all $x\in \R^n.$ Moreover,
\begin{equation}\label{zamech1}
||D_k f||_n\le \e ||D_k f_0||_n \quad (k=1,...,n)
\end{equation}
and
\begin{equation}\label{zamech2}
||f||_n\le \frac{\e ||f_0||_n}{\g}.
\end{equation}
First, this shows that $\capa(\mathcal B_r; W_n^1)=0$. Further,  applying Lemma \ref{partition} with $T=1$ and using (\ref{zamech1}) and (\ref{zamech2}), we obtain
$$
\begin{aligned}
(1-\a)^{1/q}||f||_{b_{n,q;k}^\a} &\le q^{-1/q}\left(||D_k f||_n + 2\left(\frac{1-\a}{\a}\right)^{1/q}||f||_n\right)\\
&\le \e q^{-1/q}\left(||D_k f_0||_n + \frac{2}{\g}\left(\frac{1-\a}{\a }\right)^{1/q}||f_0||_n\right).
\end{aligned}
$$
Since $\capa\left(\mathcal B_r; B_{n,q}^\a\right)\le ||f_\g||_{b_{n,q}^\a}^n,$ this implies that
$$
\varlimsup_{\a\to 0}(1-\a)^{1/q}\capa(\mathcal B_r; B_{n,q}^\a)^{1/n}\le \e q^{-1/q} \sum_{k=1}^n||D_k f_0||_n.
$$
By view of the arbitrariness of $\e>0$, we obtain (\ref{pequaln}). Thus, for $p=n\ge 2$ (\ref{main1}) also is trivially true.

\end{rem}

\begin{rem} The remaining case $p=n=1$ is also "degenerate". First, if a
set $E$ consists of one point, $E=\{x_0\},$ then $\capa(E;W_1^1)\ge
2.$ Indeed, if $f\in C_0^\infty(\R)$ and $f(x_0)=1,$ then
$$
-\int_{x_0}^\infty f'(x)\,dx=\int_{-\infty}^{x_0} f'(x)\,dx=1.
$$
Thus, $||f'||_1\ge 2.$ Further, let $K\subset \R$ be an arbitrary
compact set, $K\subset [-a,a]\,\,\,(a>0).$ Set
\begin{equation}\label{fa}
f_a(x)=\begin{cases}
    1 &\textnormal{if }\quad |x|\le a \\
    (a/x)^2 &\textnormal{if}\quad |x|>a.
  \end{cases}
\end{equation}
Then $f_a\in W_1^1(\R)$ and $||f'_a||_1=2.$ We obtain that
$\capa(K;W_1^1)=2$ for any  compact set $K\not=\emptyset$, and therefore
$\capa(G;W_1^1)=2$ for any non-empty open set $G\subset \R$.

Now we observe that for any $f\in L^1(\R)$
 and any $h>0$
 $$
 \begin{aligned}
 &\int_0^\infty |f(x)-f(x+h)|\,dx\\
 &\ge \int_0^\infty |f(x)|\,dx - \int_0^\infty |f(x+h)|\,dx=\int_0^h |f(x)|\,dx,
 \end{aligned}
 $$
 and similarly
$$
\int_{-\infty}^0 |f(x)-f(x+h)|\,dx\ge \int_0^h |f(x)|\,dx.
$$
Thus,
\begin{equation}\label{below}
\o(f;h)_1\ge 2 \int_0^h |f(x)|\,dx \quad (h>0).
\end{equation}
Let $I=[-h_0,h_0]\,\,\,(h_0>0).$ Let  $f\in L^1(\R)$ and $f(x)=1$ on
$I$. Then, by (\ref{below}), $\o(f;h)_1\ge 2h$ for all $0\le h\le
h_0.$ Thus, for any $1\le q<\infty$
$$
\begin{aligned}
(1-\a)||f||_{b_{1,q}^\a}^q&\ge (1-\a) \int_0^{h_0}h^{-\a q}\o(f;h)_1^q\frac{dh}{h}\\
&\ge 2^q(1-\a)\int_0^{h_0}h^{(1-\a) q}\frac{dh}{h} = \frac{2^q}{q}h_0^{(1-\a) q}.
\end{aligned}
$$
This implies that
$$
\varliminf_{\a\to 1-}(1-\a)^{1/q}\capa(G; B_{1,q}^\a)\ge 2q^{-1/q}
$$
for any  open set $G\subset \R$. On the other hand, assume that $G\subset
[-a,a]\,\,\,(a>0)$. Applying Lemma \ref{partition} to the
function (\ref{fa}), we have
$$
\begin{aligned}
(1-\a)^{1/q}||f_a||_{b_{1,q}^\a}&\le q^{-1/q}||f'_a||_1+2\left(\frac{1-\a}{\a q}\right)^{1/q}||f_a||_1\\
&=q^{-1/q} \left[2+8a\left(\frac{1-\a}{\a}\right)^{1/q}\right].
\end{aligned}
$$
It follows that
$$
\varlimsup_{\a\to 1-}(1-\a)^{1/q}\capa(G; B_{1,q}^\a)\le 2q^{-1/q}.
$$
Thus, for any open bounded set $G\subset \R$
$$
\lim_{\a\to 1-}(1-\a)^{1/q}\capa(G; B_{1,q}^\a)=q^{-1/q}\capa (G; W_1^1)= 2q^{-1/q}.
$$
\end{rem}

\section{The limit as $\a\to 0$}

In this section we study the behaviour of $B_{p,q}^\a-$capacities as
$\a\to 0$ (cf.  (\ref{mazya1}) and Remark \ref{final} below).
\begin{teo}\label{SECOND} Let $1\le p<\infty$ and $1\le q<\infty.$ Then for any compact set $K\subset \R^n$
\begin{equation}\label{second0}
\lim_{\a\to 0+}\a^{p/q}{\capa} (K;B_{p,q}^\a)=2n^p\left(\frac1q\right)^{p/q}|K|.
\end{equation}
\end{teo}
\begin{proof} Denote
\begin{equation*}
\Lambda(\a)=\a^{1/q}\left[\capa(K; B_{p,q}^\a\right)]^{1/p},\quad 0<\a<1.
\end{equation*}
 First we prove that
\begin{equation}\label{second00}
\varliminf_{\a\to 0+}\Lambda(\a)\ge n2^{1/p}q^{-1/q}|K|^{1/p}.
\end{equation}
We assume that $|K|>0.$  It is clear that $\varliminf_{\a\to
0+}\Lambda(\a)<\infty.$ There exists a decreasing sequence
$\{\a_\nu\}$ of numbers $\a_\nu\in (0,1/2]$  with $\a_1=\min(1,
n/p)/2$ such that $\a_\nu\to 0$ and
\begin{equation}\label{second3}
\lim_{\nu\to\infty} \Lambda(\a_\nu)=\varliminf_{\a\to 0+}\Lambda(\a).
\end{equation}
We emphasize that  $\a_\nu<n/p$ for all $\nu\in \N.$ We may assume
that
\begin{equation}\label{second4}
\Lambda(\a_\nu)\le \varliminf_{\a\to 0+}\Lambda(\a)+1 \quad (\nu\in \N).
\end{equation}

 For any $\nu\in \N$ there exists a function $f_\nu\in
C_0^\infty(\R^n)$ such that $0\le f_\nu(x)\le 1$ for all $x\in
\R^n$, $f_\nu(x)=1$ for all $x\in K,$ and
$$
\Lambda(\a_\nu)\ge \a_\nu^{1/q}||f_\nu||_{b_{p,q}^{\a_\nu}}-\frac1{\nu}.
$$
Applying Lemma \ref{BESOV1} for $\theta=\infty$, we obtain that
\begin{equation}\label{second5}
\Lambda(\a_\nu)+\frac1{\nu}\ge q^{-1/q} t^{-\a_\nu}\sum_{j=1}^n\o_j(f_\nu;t)_p
\end{equation}
for any $t>0$ and any $\nu\in \N.$ In particular, by (\ref{second4})
and (\ref{second5}),
\begin{equation}\label{second6}
\sum_{j=1}^n\o_j(f_\nu;t)_p\le At^{\a_\nu}, \quad t>0,
\end{equation}
where $A$ depends only on $p, q, n,$ and $K.$

Let $0<\e<1.$ Set
$$
E_{\nu,\e}=\{x\in\R^n: f_\nu(x)>\e\}.
$$
Denote $p_\nu=np/(n-\a_\nu p)$. Then $p_\nu\le p_1.$ By Lemma
\ref{MEASURE2},
$$
|E_{\nu,\e}|\le (2p_\nu)^{p_\nu}\e^{-p_\nu}||f_\nu||_{b_{p,\infty}^{\a_\nu}}^{p_\nu}\le (2p_1)^{p_1}\e^{-p_1}||f_\nu||_{b_{p,\infty}^{\a_\nu}}^{p_\nu}.
$$
Thus, using (\ref{second6}), we obtain that
\begin{equation}\label{second7}
|E_{\nu,\e}|\le A'\e^{-p_1} \quad (\nu\in \N),
\end{equation}
where  $A'$ depends only on $p,q,n,$ and $K.$

Since $|K|>0,$ there exists a number $\l(K)>0$ such that
$$
\mes_{n-1}
\Pi_j (K)\ge \l(K)\quad\mbox{for all}\quad 1\le j\le n.
$$
 Further, $K\subset
E_{\nu,\e}$, and thus
$$
\mes_{n-1} \Pi_j (E_{\nu,\e})\ge \l(K)\quad (\nu\in\N,\,\, j=1,...,n).
$$
 Now we apply Lemma
\ref{TRANSLATION} with $\mu=A'\e^{-p_1}$, $\l=\l(K),$ and
$\eta=\e|K|.$ Set $H=2\mu^2n/(\l\eta)$. By Lemma \ref{TRANSLATION},
for any $\nu\in\N$ there exists $h_\nu\in
(0,H]$ such that
\begin{equation}\label{second8}
\sum_{j=1}^n|\{x\in E_{\nu,\e}:x+h_\nu e_j\in E_{\nu,\e}\}|<\e |K|.
\end{equation}
We emphasize that $H$ doesn't depend on $\nu.$ Denote
$$
K_j^{(\nu)}=\{x\in \R^n: x+h_\nu e_j\in K\}.
$$
Since $K\subset
E_{\nu,\e}\,\,\, (\nu\in \N),$ we derive from (\ref{second8}) that
for any $\nu\in \N$ and any $j=1,...,n$
$$
|\{x\in K: f_\nu(x+h_\nu e_j)\le \e\}|>(1-\e)|K|,
$$
$$
  |\{x\in K_j^{(\nu)}: f_\nu(x)\le \e\}|>(1-\e)|K|,
$$
 and $|K\cap
K_j^{(\nu)}|<\e |K|$. Thus, taking into account that $0\le
f_\nu(x)\le 1$ and $f_\nu(x)=1$ on $K$, we obtain
$$
\begin{aligned}
&\o_j(f_\nu;H)_p^p\ge \int_K|f_\nu(x)-f_\nu(x+h_\nu e_j)|^p\,dx\\
 &+ \int_{K_j^{(\nu)}}|f_\nu(x)-f_\nu(x+h_\nu e_j)|^p\,dx -\e|K|
\ge 2(1-\e)^{p+1}|K|-\e|K|
\end{aligned}
$$
for all $1\le j\le n.$ From here and (\ref{second5}),
$$
\begin{aligned}
\Lambda(\a_\nu)+\frac1{\nu}&\ge q^{-1/q} H^{-\a_\nu}\sum_{j=1}^n\o_j(f_\nu;H)_p\\
&\ge nq^{-1/q} H^{-\a_\nu}[(2(1-\e)^{p+1}-\e)|K|]^{1/p}.
\end{aligned}
$$
By (\ref{second3}), this implies that
$$
\varliminf_{\a\to 0+}\Lambda(\a)\ge nq^{-1/q}[(2(1-\e)^{p+1}-\e)|K|]^{1/p}.
$$
Since $\e\in (0,1)$ is arbitrary, this yields (\ref{second00}).

Now we shall prove that
\begin{equation}\label{second000}
\varlimsup_{\a\to 0+}\Lambda(\a)\le n2^{1/p}{q^{-1/q}}|K|^{1/p}.
\end{equation}
Set for  $\tau>0$
\begin{equation}\label{second11}
K_{\tau}=\{x\in \R^n: \dist(x, K)\le 2\tau\}.
\end{equation}
Fix $0<\e<1$ and choose $\tau>0$  such that
\begin{equation}\label{second111}
|K_{\tau}|<|K|+\e.
\end{equation}
 Let $\varphi_\tau$ be the standard
mollifier defined by (\ref{moll}). Set
$$f_{\tau}=\chi_{\tau}\ast\f_\tau,$$ where $\chi_{\tau}$ is the
characteristic function of the set $K_\tau.$ Then $f_\tau\in
C_0^\infty(\R^n)$, $0\le f_\tau(x)\le 1$ for all $x\in\R^n$, and
$f_\tau(x)=1$ for all $x$ such that $\operatorname{dist} (x,K)\le
\tau.$ Thus,
\begin{equation}\label{second12}
\capa(K; B_{p,q}^\a)\le ||f_\tau||_{ b_{p,q}^\a}^p.
\end{equation}
Using (\ref{charact}) and (\ref{moll3}), we have
\begin{equation}\label{second14}
\o_j(f_\tau;t)_p\le \o_j(\chi_{\tau};t)_p\le (2|K_{\tau}|)^{1/p}\le [2(|K|+\e)]^{1/p}.
\end{equation}
Applying (\ref{ocenka}) and (\ref{second14}), we obtain
$$
\begin{aligned}
&\a^{1/q}||f_\tau||_{b_{p,q}^\a}= \a^{1/q}\sum_{j=1}^n\left(\int_0^\infty t^{-\a q}\o_j(f_\tau;t)_p^q\frac{dt}{t}\right)^{1/q}\\
&\le \a^{1/q}\sum_{j=1}^n\left[||D_j f_\tau||_p\left(\int_0^1 t^{(1-\a)q}\frac{dt}{t}\right)^{1/q}
+\left(\int_1^\infty t^{-\a q} \o_j(f_\tau;t)_p^q\frac{dt}{t}\right)^{1/q}\right]\\
&\le \left(\frac{\a}{(1-\a)q}\right)^{1/q}\sum_{j=1}^n||D_j f_\tau||_p+n2^{1/p}q^{-1/q}(|K|+\e)^{1/p}.
\end{aligned}
$$
This estimate and (\ref{second12}) imply that
$$
\Lambda(\a)
\le \left(\frac{\a}{(1-\a)q}\right)^{1/q}\sum_{j=1}^n||D_j f_\tau||_p + 2^{1/p}q^{-1/q}n(|K|+\e)^{1/p}.
$$
It follows that
$$
\varlimsup_{\a\to
0+}\Lambda(\a) \le n2^{1/p}q^{-1/q}(|K|+\e)^{1/p}.
$$
Since $\e\in (0,1)$ is arbitrary, this implies (\ref{second000}).
Inequalities (\ref{second00}) and (\ref{second000}) yield
(\ref{second0}).
\end{proof}

\begin{rem}\label{OPEN} Generally, Theorem 4.1 fails to hold for
{\it open} sets. As in Section 3, we shall show it using Cantor sets \cite[Section
5.3]{AH}.

Let $f\in B_p^\a(\R^n)$ an let $\d_\l f(x)=f(\l x)$ $(\l>0)$ be a
dilation of $f$. It is easily seen that
\begin{equation}\label{open1}
||\d_\l f||_{b_p^\a}^p=\l^{\a p-n}||f||_{b_p^\a}^p.
\end{equation}

Assume that $p>1$ and $0<\a<\min(1, n/p).$ Recall that in this case
the $B_p^\a$-capacity is equivalent to the Bessel capacity
$C_{\a,p}$ \cite[p. 107]{AH}. There exists $k_0=k_0(\a)$ such that the sequence
$$l_k=(2^{-kn}(k+k_0)^2)^{1/(n-\a p)}$$ satisfies the condition $l_{k+1}\le l_k/2 \,\,\,(k=0,1,...).$
Moreover,
$$
\sum_{k=0}^\infty 2^{-kn}l_k^{\a p-n}<\infty.
$$
Let $K_\a$ be the Cantor set corresponding to the sequence $\{l_k\}$, defined in \cite[(5.3.1)]{AH}.
Then $|K_\a|=0$ and by  \cite[Theorem 5.3.2]{AH}, $\capa(K_\a; B_p^\a)>0.$
 For $\l>0,$ set
$$
K_{\a,\l}=\{x\in\R^n: \frac{x}{\l}\in K_\a\}.
$$
There exists a function $f_{\a,\l}\in C_0^\infty$ such that $0\le
f_{\a,\l}(x)\le 1$ for all $x\in\R^n,$ $f_{\a,\l}(x)=1$ in some neighborhood of
$K_{\a,\l}$, and
$$
||f_{\a,\l}||_{b_p^\a}^p\le \capa(K_{\a,\l}; B_p^\a)+1.
$$
Set $g_{\a,\l}(x)=f_{\a,\l}(\l x).$ Then $g_{\a,\l}(x)=1$ in some neighborhood of
$K_{\a}$. Thus, using (\ref{open1}), we obtain
$$
\begin{aligned}
\capa(K_\a; B_p^\a)&\le ||g_{\a,\l}||_{b_p^\a}^p=\l^{\a p-n}||f_{\a,\l}||_{b_p^\a}^p\\
&\le \l^{\a p-n}(\capa(K_{\a,\l}; B_p^\a)+1).
\end{aligned}
$$
From here,
$$
\capa(K_{\a,\l}; B_p^\a)\ge \l^{n-\a p}\capa(K_\a; B_p^\a) - 1.
$$
Since $\capa(K_\a; B_p^\a)>0,$ we can choose such $\l(\a)>0$  that
$$
\a\capa(K_{\a,\l(\a)}; B_p^\a)>1.
$$
Thus, for
any $0<\a<\min(1, n/p)$ there exists a compact set $E_\a$ such that
$$
|E_\a|=0 \quad \mbox{and}\quad \a\capa(E_\a; B_p^\a)>1.
$$
Let $j_0=[(\min(1, n/p))^{-1}]+1.$ Set
$E^*_j=E_{1/j}, \,\, j\ge j_0.$ Then
$$
\a \capa(E^*_j; B_p^\a)>1 \quad\mbox{for} \quad \a =\frac1j\,\, (j\ge j_0).
$$
Further, set  $E=\cup_{j=j_0}^\infty E^*_j.$
Then $|E|=0$.
Let $0<\e<1$.
There exists an open set $G$ such that $E\subset G$ and $|G|<\e.$ We have
$$
\a \capa(G; B_p^\a)\ge \a \capa(E^*_j; B_p^\a)>1 \quad\mbox{for} \quad \a =\frac1j\quad (j\ge j_0).
$$
Thus,
$$
\varlimsup_{\a\to 0}\a \capa(G; B_p^\a)\ge 1,
$$
and equality (\ref{second0}) does not hold for the set $G$.

\end{rem}

\begin{rem}\label{final} Our final remark concerns limiting relation (\ref{mazya1}). This relation was proved in \cite{MS}
for the seminorm
$$
\left(\int_{\R^n}\int_{\R^n} \frac{|f(x+h)-f(x)|^p}{|h|^{n+\a p}}\,dxdh.\right)^{1/p}.
$$
It is well known that this  seminorm is  equivalent to $||f||_{b_p^\a}.$
We shall briefly discuss the limiting behaviour of  $\a||f||_{b_p^\a}$.

Assume that a function $f$ belongs to $B_{p,q}^{\a_0}(\R^n)$ for some $0<\a_0<1.$ Then $f\in B_{p,q}^{\a}(\R^n)$
for any $0<\a\le \a_0.$ Moreover, it follows immediately from \cite[Lemma 1]{KMX} that
\begin{equation}\label{final1}
\lim_{\a\to  0}\a^{1/q} \|f\|_{b_{p,q}^{\a}}=q^{-1/q}\sum_{j=1}^n\o_j(f;+\infty)_p.
\end{equation}
 It is also easily seen that
$$
\lim_{h\to \infty} \int_{\R^n} |f(x+he_j)-f(x)|^p\,dx= 2 ||f||_p^p\quad (j=1,...,n).
$$
This equality and (\ref{posit}) imply that for a {\it nonnegative} $f$
\begin{equation}\label{final2}
\o_j(f;+\infty)=2^{1/p}||f||_p\quad (j=1,...,n)
\end{equation}
and thus by (\ref{final1})
\begin{equation}\label{final3}
\lim_{\a\to  0}\a^{1/q} \|f\|_{b_{p,q}^{\a}}=q^{-1/q}2^{1/p}n||f||_p\quad \mbox{if}\quad f\ge 0.
\end{equation}
However,  equalities (\ref{final2}) and (\ref{final3}) fail to hold in a general case. We consider the following simple example for $n=1.$ Let
$I_k=[k,k+1)$ $(k=0,1,...,2\nu).$ Set
$$
f_\nu(x)=\sum_{k=0}^{2\nu} (-1)^k\chi_{I_k}(x).
$$
Then $||f_\nu||_p=(2\nu+1)^{1/p}.$ Further,
$$
\begin{aligned}
&\int_\R |f_\nu(x+1)-f_\nu(x)|^p\,dx\\
&\ge \sum_{k=0}^{2\nu-1} \int_{I_k}|f_\nu(x+1)-f_\nu(x)|^p\,dx=2^{p+1}\nu.
\end{aligned}
$$
Thus,
$$
\o(f_\nu;1)_p\ge 2\left(\frac{2\nu}{2\nu+1}\right)^{1/p}||f_\nu||_p.
$$
It shows that the constant 2 on the right-hand side of (\ref{First})
is optimal, and thus  (\ref{final2}) and (\ref{final3}) may not be true.

\end{rem}


\begin{thebibliography}{99}

\bibitem{Ad1} D.R. Adams, {\it The classification problem for the
capacities associated with the Besov and Triebel-Lizorkin spaces,}
Banach Center Publ. {\bf 22} (1989), 9 -- 24.




\bibitem{AH} D.R. Adams and L.I. Hedberg, {\it Function Spaces and Potential Theory,} Springer-Verlag Berlin Heidelberg New York 1999.



\bibitem{BS} C. Bennett and R. Sharpley,  {\it Interpolation of
Operators,} Academic Press, Boston 1988.


\bibitem{BIN} O.V. Besov, V.P. Il'in and S.M. Nikol'ski\u\i,
{\it Integral Representations of Functions and Embedding Theorems},
Vol. 1-2, Winston, Washington, D.C., Halsted Press, New York, 1978
-- 1979.


\bibitem{BBM1}
J. Bourgain, H. Brezis and P. Mironescu, {\it Another look at
Sobolev spaces}, Optimal Control and Partial Differential
Equations. In honour of Professor Alain Bensoussan's 60th
Birthday. J. L. Menaldi, E. Rofman, A. Sulem (eds), IOS Press,
Amsterdam, 2001, 439 -- 455.

\bibitem{BBM2}
J. Bourgain, H. Brezis and P. Mironescu, {\it Limiting Embedding
Theorems for $W^{s,p}$ when $s\uparrow 1$ and applications}, J.
D'Analyse Math. {\bf 87} (2002), 77 -- 101.





\bibitem{Br1}  H. Brezis, {\it How to recognize constant functions.
Connections with Sobolev spaces,} (Russian) Uspekhi Mat. Nauk {\bf
57} (2002), no. 4(346), 59--74; English transl. in Russian Math.
Surveys {\bf 57} (2002), no. 4, 693--708.



\bibitem{Br2} H. Brezis, {\it Functional Analysis, Sobolev Spaces and Partial Differential Equations,}
Springer New York Doldrecht Heidelberg London 2011.






\bibitem{CGO} A.M. Caetano, A. Gogatishvili, B. Opic, {\it  Embeddings and the growth envelope of Besov spaces involving only slowly varying smoothness,} J. Approx. Theory, {\bf 163} (2011),
1373 -- 1399.





\bibitem {KMX} G.E. Karadzhov, M. Milman and J. Xiao, {\it Limits of higher order Besov spaces and sharp reiteration theorems,} J. Funct. Anal. {\bf 221} (2005), 323 -- 339.

\bibitem{K1975}
V.I. Kolyada, {\it On imbedding in  classes $\varphi(L),$} Izv.
Akad. Nauk SSSR Ser. Mat. {\bf 39} (1975), 418 -- 437; English
transl.:  Math. USSR Izv. {\bf 9} (1975), 395 -- 413.

\bibitem{K1988}
V.I. Kolyada, {\it Estimates of rearrangements and embedding
theorems}, Mat. Sb. {\bf 136} (1988), 3 -- 23; English transl.:
Math. USSR-Sb. {\bf 64} (1989), no. 1, 1 -- 21.

\bibitem{K1989}
V.I. Kolyada, {\it Rearrangements of functions and embedding
theorems,} Uspekhi matem. nauk {\bf 44} (1989), no. 5, 61 -- 95;
English transl. in Russian Math. Surveys {\bf 44} (1989), no. 5, 73
-- 118.

\bibitem{K1998}
V.I. Kolyada, {\it Rearrangement of functions and embedding of
anisotropic spaces of Sobolev type}, East J. Approx. {\bf 4} (1998),
no. 2, 111 -- 199.



\bibitem{K2006}  V.I. Kolyada, {\it Mixed norms and Sobolev type
inequalities}, Banach Center Publ. {\bf 72} (2006), 141 -- 160.










\bibitem {K2007}  V.I. Kolyada, {\it On embedding theorems}, in:
Nonlinear Analysis, Function Spaces and Applications, vol. 8
(Proceedings of the Spring School held in Prague, 2006), Prague,
2007, 35 - 94.

\bibitem{KL}
V.I. Kolyada and A.K. Lerner, {\it On limiting embeddings of Besov
spaces}, Studia Math. {\bf 171}, no. 1 (2005), 1 -- 13.



\bibitem{Leo1} G. Leoni, {\it A First Course in Sobolev Spaces,}
Graduate Studies in Mathematics, v. 105, 2009.



\bibitem{LS} G. Leoni, D. Spector, {\it Characterization of
Sobolev and BV spaces}, J. Funct. Anal. {\bf 261} (2011), 2926 --
2958.




\bibitem{Maz2}
 V. Maz'ya, {\it Sobolev spaces with applications to elliptic partial differential equations,} Second, revised and augmented edition.
 Grundlehren der Mathematischen Wissenschaften, 342. Springer, Heidelberg, 2011.

\bibitem{MS}
V. Maz'ya and T. Shaposhnikova, {\it On the Bourgain, Brezis, and
Mironescu theorem concerning limiting embeddings of fractional
Sobolev spaces}, J. Funct. Anal. {\bf 195} (2002), no. 2, 230 --
238.

\bibitem{Mil} M. Milman, {\it Notes on limits of Sobolev spaces and the continuity of interpolation scales,} Trans. Amer. Math. Soc. {\bf 357} (2005), no. 9, 3425 –- 3442.


\bibitem{Nik} S.M. Nikol'ski\u\i, {\it Approximation of Functions of Several
 Variables and Imbedding Theorems,} Springer -- Verlag, Berlin -- Heidelberg -- New York,
 1975.


\bibitem{St} E.M. Stein, {\it Singular Integrals and
Differentiability Properties of Functions}, Princeton Univ.
Press,~1970.









\bibitem{Tri} H. Triebel, {\it Limits of Besov norms,} Arch. Math.
{\bf 96} (2011), 169 -- 175.


\bibitem{U} P.L. Ul'yanov, {\it Embedding of certain function classes
$H_p^\omega$}, Izv. Akad. Nauk SSSR Ser. Mat. {\bf 32} (1968), 649
-- 686; English transl. in Math. USSR Izv. {\bf 2} (1968), 601 --
637.





\end{thebibliography}
\end{document}